\documentclass{amsart}
\usepackage{amsmath,amsthm,amssymb}
\usepackage[all]{xy}

\newcommand{\gA}{\mathcal{A}}
\newcommand{\gM}{\mathcal{M}}

\newcommand{\gG}{\mathcal{G}}
\newcommand{\gS}{\mathcal{S}}
\newcommand{\cI}{\mathcal{I}}
\newcommand{\LiegG}{\mathfrak{g}}
\newcommand{\pinfty}{P_\infty}
\newcommand{\linfty}{L_\infty}
\newcommand{\Z}{\mathbb{Z}}
\newcommand{\N}{\mathbb{N}}

\newcommand{\shiftgA}{\tilde{\mathcal{A}}}
\newcommand{\vect}{\mathfrak{X}}
\newcommand{\multivect}{\vect^\bullet}

\DeclareMathOperator{\bigS}{S}
\newcommand{\linv}[1]{\stackrel{\leftarrow}{#1}}
\newcommand{\rinv}[1]{\stackrel{\rightarrow}{#1}}
\renewcommand{\hat}[1]{\widehat{#1}}
\renewcommand{\tilde}[1]{\widetilde{#1}}
\newcommand{\dg}{\mathrm{dg}}

\newcommand{\bitimes}[2]{\,\vphantom{\times}_{#1} \! \times_{#2}}
\newcommand{\llbracket}{[\![}
\newcommand{\rrbracket}{]\!]}

\newcommand{\suchthat}{:}
\newcommand{\defequal}{:=}

\newcommand{\comment}[1]{}

\newtheorem{thm}{Theorem}[section]
\newtheorem{prop}[thm]{Proposition} 
\newtheorem{lemma}[thm]{Lemma}
\newtheorem{cor}[thm]{Corollary}

\theoremstyle{definition} 
\newtheorem{dfn}[thm]{Definition}

\theoremstyle{remark}
\newtheorem{remark}[thm]{Remark}
\newtheorem{remarks}[thm]{Remarks}
\newtheorem{examples}[thm]{Examples}
\newtheorem{example}[thm]{Example}

\newtheorem*{ack}{Acknowledgements}
\numberwithin{equation}{section}

\title{On homotopy Poisson actions and reduction of symplectic $Q$-manifolds}
\author{Rajan Amit Mehta}

\begin{document}

\begin{abstract}
 We present a general framework for reduction of symplectic $Q$-manifolds via graded group actions. In this framework, the homological structure on the acting group is a multiplicative multivector field.
\end{abstract}

\maketitle

\section{Introduction} 

A degree $n$ symplectic $NQ$-manifold is an $\N$-graded manifold equipped with a degree $n$ symplectic structure and a symplectic homological vector field. There has recently been much interest in symplectic $NQ$-manifolds, primarily arising from the following one-to-one correspondences, due to Roytenberg \cite{royt:graded} and Severa \cite{severa:sometitle}:
\begin{align*}
\fbox{degree $1$ symplectic $NQ$-manifolds} &\longleftrightarrow \fbox{Poisson manifolds}\\
\fbox{degree $2$ symplectic $NQ$-manifolds} &\longleftrightarrow \fbox{Courant algebroids}
\end{align*}
These correspondences suggest that one should be able to study Poisson, Courant, and generalized complex geometries from the perspective of graded symplectic geometry. Following this premise, the papers \cite{bcmz} and \cite{cz:poisson} have shown that various constructions for Poisson, Courant, and generalized complex reduction can be interpreted as examples of symplectic reduction of $NQ$-manifolds. This point of view unifies the various methods of reduction and, in a sense, explains why they work.

Consider moment map reduction in the case where $\gS$ is a symplectic $NQ$-manifold and $\gG$ is a graded Lie group with a Hamiltonian action on $\gS$. Under the usual regularity assumptions, a graded version of the Marsden-Weinstein theorem guarantees that the reduced space inherits a symplectic structure. In order to ensure that the homological vector field on $\gS$ also descends to the quotient, one must impose additional hypotheses. In \cite{cz:poisson} and \cite{bcmz}, it was proposed (in the cases $n=1$ and $n=2$, respectively) that $\gG$ should be allowed to have extra structure; namely, the Lie algebra $\LiegG$ of $\gG$ should be a differential graded Lie algebra (DGLA). A DGLA structure on $\LiegG$ integrates to a multiplicative homological vector field on $\gG$, making $\gG$ into a \emph{$\dg$-group} or a \emph{$Q$-group}.  It was shown in \cite{cz:poisson} and \cite{bcmz} (again in their respective cases) that, if $\gG$ is connected and the comoment map $\LiegG[n] \to C^\infty(\gS)$ is a morphism of DGLAs, then the vector field on $\gS$ passes to the symplectic quotient. 

On the other hand, the following example of Poisson quotients illustrates that $Q$-groups or DGLAs are insufficiently general to incorporate all known notions of reduction. Let $M$ be a Poisson manifold, and let $G$ be a Poisson Lie group with a Poisson action on $M$. Then it is known that the quotient $M/G$, if smooth (for example, if the action is free and proper), inherits a Poisson structure. This example may be restated in the language of graded symplectic geometry as follows. The degree $1$ symplectic $NQ$-manifold associated to the Poisson manifold $M$ is the shifted cotangent bundle $T^*[1]M$. The action of $G$ on $M$ naturally lifts to a Hamiltonian action on $T^*[1]M$, and the symplectic quotient may be canonically identified with $T^*[1](M/G)$. The homological vector field on $T^*[1]M$ naturally descends to $T^*[1](M/G)$, giving the Poisson structure on $M/G$. 

The example that we have just described involves the action of a Poisson Lie group and does not fit into the framework of reduction by $Q$-group or DGLA action. This fact suggests that the theory of reduction of symplectic $NQ$-manifolds should allow the acting group to possess a general structure that includes both $Q$-groups and Poisson Lie groups as special cases. The search for such a generalization leads us to the notion of \emph{homotopy Poisson Lie group}. 

A homotopy Poisson structure is an $\linfty$ structure whose brackets satisfy a Leibniz rule. Such structures have appeared in the work of Voronov \cite{voronov:higher1,voronov:higher2} under the name ``higher Poisson brackets,''
and Cattaneo and Felder \cite{cf:formality} have called them $\pinfty$ structures. See \cite{bruce:higher} for more recent work involving such structures.

A homotopy Poisson Lie group is a graded Lie group whose algebra of functions has a multiplicative homotopy Poisson algebra structure. Both $Q$-groups and Poisson Lie groups are examples of homotopy Poisson Lie groups. The main result of this paper is a reduction result for homotopy Poisson Lie group actions---essentially, if a homotopy Poisson Lie group has a Hamiltonian action on a symplectic $NQ$-manifold, then under certain compatibility and regularity conditions, the symplectic quotient inherits a homological vector field. This result provides a quite general framework for moment map reduction of symplectic $NQ$-manifolds.

The structure of the paper is as follows. In \S\ref{sec:pinftyman}, we review the definition of homotopy Poisson manifolds and their characterization in terms of multivector fields. In \S\ref{sec:pinftylie}, we introduce homotopy Poisson Lie groups and describe their corresponding infinitesimal objects, which we call homotopy Lie bialgebras. We consider actions of homotopy Poisson Lie groups on homotopy Poisson manifolds in \S\ref{sec:pinftyact}, proving that if an action satisfies a compatibility condition, then its quotient inherits a homotopy Poisson structure. The heart of the paper is \S\ref{sec:hamiltonian}, where we turn to Hamiltonian actions of homotopy Poisson Lie groups on degree $1$ symplectic (not necessarily $N$)$Q$-manifolds. In \S\ref{sec:pinftyact2}, we show how the framework of \S\ref{sec:hamiltonian} includes homotopy Poisson quotients and, in particular, Poisson quotients. Finally, in \S\ref{sec:higher}, we describe a modification of the definition of homotopy Poisson Lie group that allows one to extend the results of \S\ref{sec:hamiltonian} to reduction of degree $n$ symplectic $Q$-manifolds.

\begin{ack}
 The author thanks Marco Zambon, Florian Sch\"{a}tz, Henrique Bursztyn, and Alberto Cattaneo for helpful comments and discussions.
\end{ack}

\section{Homotopy Poisson manifolds}\label{sec:pinftyman}
In this section, we recall the notions of homotopy Poisson algebra and homotopy Poisson manifold. For more details and applications of such structures, we refer the reader to \cite{cf:formality, cat:reduction, schatz:thesis} (under the name \emph{$\pinfty$-manifolds} and \cite{voronov:higher1, voronov:higher2, bruce:higher} (under the name \emph{higher Poisson manifolds}).

Let $\gA = \sum_{i \in \Z} \gA_i$ be a graded commutative algebra over a field $k$, and let $\shiftgA \defequal \gA[1]$ be the degree $1$ suspension of $\gA$, so that $\shiftgA_i = \gA_{i+1}$.  
\begin{dfn}\label{dfn:palg}
A \emph{homotopy Poisson algebra structure} on $\gA$ consists of a series of multilinear $\ell$-ary brackets $\beta_\ell : \shiftgA^{\otimes \ell} \to \shiftgA$ for $\ell \geq 0$ such that the following properties hold:
\begin{enumerate}
\item The degree of $\beta_\ell$ is $1$; that is, $|\beta_\ell(a_1, \dots, a_\ell)| = 1 + \sum |a_i|$ for all homogeneous $a_1, \dots, a_\ell \in \shiftgA$.
\item The brackets $\beta_\ell$ are graded symmetric.
\item If $a_1, \dots, a_{\ell-1}$ are homogeneous elements of $\shiftgA$, then $\beta_\ell(a_1, \dots, a_{\ell-1}, \cdot)$ is a graded derivation of degree $1 + \sum |a_i|$.
\item The brackets $\beta_\ell$ satisfy the generalized Jacobi identities (with the sign convention of \cite{voronov:higher1}; also see \cite{schatz:thesis}).
\end{enumerate}
\end{dfn}
In other words, a homotopy Poisson algebra is a graded commutative algebra with an $\linfty$ structure such that the multibrackets satisfy a Leibniz rule. We may think of a homotopy Poisson structure as being a generalization of a Poisson structure, where the Jacobi identity is allowed to only hold ``up to homotopy.'' We remark, however, that homotopy Poisson structures do not give the cofibrant resolution of the Poisson operad, where the Leibniz rule, as well as the commutativity and associativity of multiplication, would be similarly weakened.	

\begin{remarks}
\begin{enumerate}
\item The traditional definition of $\linfty$-algebra does not include a $0$-ary bracket $\beta_0$. Following common terminology, we say that a homotopy Poisson algebra is \emph{flat} if $\beta_0$ vanishes.
\item If $\beta_\ell$ is thought of as a bracket on $\gA$, as opposed to $\shiftgA$, then $\beta_\ell$ is of degree $2-\ell$.
\item We will say that a homotopy Poisson algebra is \emph{of finite type} if there exists a $q$ such that $\beta_\ell = 0$ for all $\ell > q$.	
\end{enumerate}
\end{remarks}

Let $\gA$ and $\gA'$ be homotopy Poisson algebras.  A (strict) morphism of homotopy Poisson algebras from $\gA$ to $\gA'$ is a map $\eta: \gA \to \gA'$ such that $\eta(\beta_\ell(a_1, \dots, a_\ell)) = \beta^\prime_\ell(\eta(a_1), \dots, \eta(a_\ell))$ for all $a_1, \dots, a_\ell \in \gA$. One could also try to define weak morphisms, but this will be unnecessary for the present purposes.

\begin{dfn}
A \emph{homotopy Poisson manifold} is a $\Z$-graded manifold $\gM$ whose algebra of functions $C^\infty(\gM)$ is equipped with a homotopy Poisson algebra structure of finite type. A map $\psi:\gM \to \gM'$ of homotopy Poisson manifolds is a homotopy Poisson morphism if the pullback map $\psi^*: C^\infty(\gM') \to C^\infty(\gM)$ is a morphism of homotopy Poisson algebras.
\end{dfn}

  Let $(\gM, \beta_\ell)$ be a homotopy Poisson manifold.  Properties (1)-(3) in Definition \ref{dfn:palg} imply that for each $\ell$ there exists an $\ell$-vector field $\pi_\ell \in \vect^\ell (\gM)$ of degree $2 - \ell$ such that $\beta_\ell$ is given by the derived bracket formula
\begin{equation}\label{eqn:derived}
	\beta_\ell(f_1, \dots, f_\ell) = [[ \cdots [\pi_\ell, f_1], \cdots ], f_\ell]
\end{equation}
for $f_1, \dots, f_\ell \in C^\infty(\gM)$.  

The algebra $\multivect(\gM)$ of multivector fields has a natural bigrading, and with respect to the total grading the $\ell$-vector fields $\pi_\ell$ are all of degree $2$.
If we let $\pi = \sum \pi_\ell$, then the generalized Jacobi identities are equivalent to the equation 
\begin{equation}\label{eqn:integrability}
[\pi,\pi]=0.  	
\end{equation}
A degree $2$ multivector field satisfying \eqref{eqn:integrability} has been called a \emph{Poisson multivector field} by Sch\"{a}tz \cite{schatz:thesis}. The derived bracket \eqref{eqn:derived} relating Poisson multivector fields and homotopy Poisson structures is due to Voronov \cite{voronov:higher1, voronov:higher2}.

The following statement summarizes the above discussion.

\begin{prop}
 Let $\gM$ be a $\Z$-graded manifold. The derived bracket formula \eqref{eqn:derived} gives a one-to-one correspondence between homotopy Poisson structures and Poisson multivector fields on $\gM$.
\end{prop}

We leave the proof of the following proposition as an exercise for the reader.
\begin{prop}
If $\gM$ and $\gM'$ are homotopy Poisson manifolds with Poisson multivector fields $\pi$ and $\pi'$, respectively, then a smooth map $\psi: \gM \to \gM'$ is a homotopy Poisson morphism if and only if $\pi$ is $\psi$-related to $\pi'$.
\end{prop}

\begin{examples}
\begin{enumerate}
\item A homotopy Poisson manifold where $\beta_\ell = 0$ for $\ell \neq 1$ is a \emph{$Q$-manifold}, that is a graded manifold equipped with a degree $1$ vector field $\pi_1$ such that $(\pi_1)^2 = 0$; such a vector field is called \emph{homological}. Many interesting examples of $Q$-manifolds come from Lie algebroids \cite{vaintrob}.
\item A homotopy Poisson manifold where $\beta_\ell = 0$ for $\ell \neq 2$ is a graded Poisson manifold. In particular, a homotopy Poisson structure on an ordinary manifold is the same thing as a Poisson structure, since a degree $2$ multivector field on an ordinary manifold is necessarily a bivector field.
\item A homotopy Poisson manifold where $\beta_\ell$ is nonzero for only $\ell = 1,2$ is a \emph{$QP$-manifold}, that is a graded Poisson manifold equipped with a homological Poisson vector field. This type of structure appears, for example, in BRST quantization \cite{henneaux-teitelboim}.
\item By allowing $\beta_\ell$ to be nonzero for more than two values, one obtains more general possibilities. For example, if $\beta_\ell$ is nonzero for only $\ell = 0,1,2$, then equation \eqref{eqn:integrability} decomposes into the following set of properties:
\begin{itemize}
\item $\pi_2$ is a Poisson bivector.
\item $\pi_1$ is a Poisson vector field.
\item $\pi_0$ is a $\pi_1$-invariant function.
\item The failure of $\pi_1$ to be homological, $(\pi_1)^2$, equals the Hamiltonian vector field of $\pi_0$.
\end{itemize}
\end{enumerate}
\end{examples}

Let $\gM$ be a homotopy Poisson manifold with Poisson multivector field $\pi$.  Define the operator $d_\pi \defequal [\pi, \cdot]$ on $\multivect(\gM)$. This operator is of total degree $1$ since $\pi$ is of total degree $2$, and the integrability equation \eqref{eqn:integrability} holds if and only if $d_\pi^2 = 0$.

\begin{prop}\label{prop:correspondence}
 The map $\pi \mapsto d_\pi$ gives a bijection from Poisson multivector fields on a graded manifold $\gM$ to degree $1$ differentials on $\multivect(\gM)$ that are derivations of the Schouten bracket.
\end{prop}

\begin{proof}
Let $\delta$ be a degree $1$ operator on $\multivect(\gM)$ that is a derivation of the Schouten bracket and such that $\delta^2=0$. Write $\delta = \sum \delta_\ell$, where $\delta_\ell$ takes elements of $\vect^i(\gM)$ to $\vect^{i + \ell -1}(\gM)$. For $\ell > 0$, we can then obtain an $\ell$-ary bracket $\beta^\delta_\ell$ by the formula
\begin{equation*}
 \beta^\delta_\ell(f_1, \dots, f_\ell) = [[ \cdots [\delta_\ell f_1, f_2], \cdots ], f_\ell].
\end{equation*}
The skew-symmetry and Leibniz rule for $\beta^\delta_\ell$ follow from the derivation property of $\delta_\ell$. If we define an $\ell$-vector field $\pi^\delta_\ell$ by the derived bracket equation
\begin{equation*}
	\beta^\delta_\ell(f_1, \dots, f_\ell) = [[ \cdots [[\pi^\delta_\ell, f_1],f_2], \cdots ], f_\ell],
\end{equation*}
then we immediately have $\delta_\ell = [\pi^\delta_\ell, \cdot]$.

To deal with the case $\ell = 0$, we use the Euler vector field $\varepsilon \in \vect(\gM)$ associated to the grading on $\gM$, defined by $\varepsilon(f) = |f| f$ for a function $f \in C^\infty(\gM)$ of homogeneous degree $|f|$. Then we let $\pi^\delta_0 = -\frac{1}{2}\delta_0 \varepsilon$.

The reader may verify that the map $\delta \mapsto \pi^\delta$ that we have described is a two-sided inverse to the map $\pi \mapsto d_\pi$.
\end{proof}

Recall that, if $\gM$ is a graded manifold, then $\multivect(\gM)$ is the algebra of polynomial functions on $T^*[1]\gM$. From this point of view, the Schouten bracket of multivector fields can be identified with the Poisson bracket associated to the canonical degree $1$ symplectic structure on $T^*[1]\gM$. Thus we may view Proposition \ref{prop:correspondence} as giving a one-to-one correspondence between $P_\infty$ structures on $\gM$ and (polynomial) symplectic homological vector fields on $T^*[1]\gM$.

\begin{remark}
 If $\gM$ is a homotopy Poisson manifold, then $T^*[1]\gM$ is a degree $1$ symplectic $Q$-manifold but in general will have coordinates in negative degrees even if $\gM$ is nonnegatively graded. Therefore it is essential for the purposes of this paper that we consider symplectic $Q$-manifolds that are $\Z$-(as opposed to $\N$-)graded.
\end{remark}

\section{Homotopy Poisson Lie groups}\label{sec:pinftylie}
\begin{dfn}
A \emph{homotopy Poisson Lie group} is a graded Lie group $\gG$ equipped with a homotopy Poisson manifold structure such that the multiplication map $m: \gG \times \gG \to \gG$ is a homotopy Poisson morphism.	
\end{dfn}
Examples of homotopy Poisson Lie groups include $Q$-groups (that is, graded Lie groups equipped with multiplicative homological vector fields) and Poisson Lie groups.

\begin{dfn} \label{dfn:leftinv}
Let $\gG$ be a graded Lie group.  A multivector field $X \in \multivect(\gG)$ is \emph{left-invariant} if $(0,X) \in \multivect(\gG \times \gG)$ is $m$-related to $X$; it is \emph{right-invariant} if $(X,0)$ is $m$-related to $X$.
\end{dfn}

The following statement is an analogue of the fact that for Poisson Lie groups, the Poisson bivector acts on left-invariant vector fields (see \cite{lu-wei}).
\begin{prop} \label{prop:pdgla}
 Let $\gG$ be a homotopy Poisson Lie group with Poisson multivector field $\phi$.  If $X \in \multivect(\gG)$ is a left-invariant multivector field, then $[\phi,X]$ is also left-invariant; if $X$ is right-invariant, then $[\phi,X]$ is also right-invariant.
\end{prop}
\begin{proof}
By the definition of homotopy Poisson Lie group, we have that $(\phi,\phi)$ is $m$-related to $\phi$. Relatedness is preserved by the Schouten bracket, so if $X$ is left-invariant then $[(\phi,\phi), (0,X)]$ is $m$-related to $[\phi,X]$. On the other hand, $[(\phi,\phi), (0,X)] = ([\phi,0],[\phi,X]) = (0, [\phi,X])$. Therefore $[\phi,X]$ satisfies the left-invariance property of Definition \ref{dfn:leftinv}. The proof for right-invariant multivector fields is similar.
\end{proof}

It follows from Proposition \ref{prop:pdgla} that the differential $d_\phi \defequal [\phi, \cdot]$ on $\multivect(\gG)$ may be restricted to the algebra of left-invariant multivector fields. If $\LiegG$ is the graded Lie algebra of $\gG$, then we can identify $\bigS(\LiegG [-1])$ with the algebra of left-invariant multivector fields\footnote{We use $\bigS(\LiegG [-1])$ instead of $\bigwedge \LiegG$ in order to agree with our sign convention for multivector fields, where $\multivect(\gM) = \bigS(\vect(\gM)[1])$.}. We denote the restriction of $d_\phi$ to $\bigS(\LiegG [-1])$ by $\hat{d}_\phi$. Thus we have the following infinitesimal description of homotopy Poisson Lie groups.

\begin{dfn}\label{dfn:pinftylie}
 A \emph{homotopy Lie bialgebra} is a graded Lie algebra $\LiegG$ equipped with a differential $\hat{d}_\phi$ on $\bigS(\LiegG [-1])$ that is a derivation of the Schouten bracket.
\end{dfn}

\begin{remark}
Definition \ref{dfn:pinftylie} may be equivalently formulated as follows: A homotopy Lie bialgebra is a graded Lie algebra $\LiegG$ such that $\LiegG^*[1]$ is equipped with a polynomial homological Poisson vector field $\hat{d}_\phi$. Any polynomial homological vector field on $\LiegG^*[1]$ is equivalent to an $L_\infty$-algebra structure on $\LiegG^*$, and the requirement that $\hat{d}_\phi$ be Poisson expresses a compatibility between the Lie algebra structure on $\LiegG$ and the $L_\infty$ structure on $\LiegG^*$ associated to $\hat{d}_\phi$.
\end{remark}

Examples of homotopy Lie bialgebras include differential graded Lie algebras (DGLAs) and Lie bialgebras. Indeed, these are the infinitesimal objects corresponding to $Q$-groups and Poisson Lie groups, respectively.

\section{Homotopy Poisson actions}\label{sec:pinftyact}
Let $\gM$ be a homotopy Poisson manifold with Poisson multivector field $\pi$, and let $\gG$ be a homotopy Poisson Lie group with Poisson multivector field $\phi$.
\begin{dfn}
A \emph{homotopy Poisson action} of $\gG$ on $\gM$ is a graded group action such that the action map $\sigma: \gM \times \gG \to \gM$ is a homotopy Poisson morphism.	
\end{dfn}
The following proposition gives an infinitesimal description of homotopy Poisson actions, generalizing that of Poisson actions given by Lu and Weinstein \cite{lu-wei}.
\begin{prop}\label{prop:infaction}
 Let $\sigma: \gM \times \gG \to \gM$ be a right homotopy Poisson action, and let $\rho: \LiegG \to \vect(\gM)$ be the induced infinitesimal action.  Then the map $\hat{\rho} \defequal \bigS(\rho[-1]) : \bigS(\LiegG [-1]) \to \multivect(\gM)$ is a morphism of differential Gerstenhaber algebras.
\end{prop}
\begin{proof}
That $\hat{\rho}$ respects brackets follows, as in the case of ordinary Lie group actions, from the fact that $\sigma$ is a right action.  It remains to show that $\hat{\rho}$ respects differentials.

The map $\hat{\rho}$ is uniquely characterized by the property that $(0,\linv{v}) \in \multivect(\gM \times \gG)$ is $\sigma$-related to $\hat{\rho}(v)$ for any $v \in \bigS(\LiegG[-1])$, where $\linv{v}$ is the left-invariant multivector field associated to $v$.  

By the definition of homotopy Poisson action, we have that $(\pi,\phi)$ is $\sigma$-related to $\pi$. Since the Schouten bracket preserves relatedness, we then have that $[(\pi,\phi), (0,\linv{v})]$ is $\sigma$-related to $[\pi, \hat{\rho}(v)]$. On the other hand, $[(\pi,\phi), (0,\linv{v})] = (0, [\phi,\linv{v}]) = (0, d_\phi\!\! \linv{v})$, so we conclude that $\hat{\rho}(\hat{d}_\phi v) = [\pi, \hat{\rho}(v)] = d_\pi \hat{\rho}(v)$.
\end{proof}

\begin{lemma}\label{lemma:quotient}
 Let $\gM$ be a homotopy Poisson manifold, and let $\gG$ be a homotopy Poisson Lie group with a homotopy Poisson action on $\gM$.  Then the algebra $C^\infty(\gM)^\gG$ of $\gG$-invariant functions on $\gM$ is closed under the multibrackets $\beta_\ell$.  Therefore, if the action is free and proper (so that the quotient $\gM/\gG$ is a graded manifold and we may identify $C^\infty(\gM)^\gG$ with $C^\infty(\gM/\gG)$), then $\gM/\gG$ inherits a homotopy Poisson structure.
\end{lemma}
\begin{proof}
Let $f_1, \dots, f_\ell \in C^\infty(\gM)$ be $\gG$-invariant. Equivalently, $\tau^* f_i = \sigma^* f_i$, where $\tau: \gM \times \gG \to \gM$ is the map given by projection onto $\gM$. Clearly, $\tau$ is a homotopy Poisson morphism, so, letting $\tilde{\beta}_\ell$ denote the $\ell$-ary bracket on $C^\infty(\gM \times \gG)$, we have
\begin{equation*}
	\begin{split}
		\tau^* \beta_\ell(f_1, \dots, f_\ell) &= \tilde{\beta}_\ell(\tau^*f_1, \dots, \tau^*f_\ell) \\
&= \tilde{\beta}_\ell(\sigma^*f_1, \dots, \sigma^*f_\ell) \\
&= \sigma^* \beta_\ell(f_1, \dots, f_\ell).
	\end{split}
\end{equation*}
Thus, $\beta_\ell(f_1, \dots, f_\ell)$ is $\gG$-invariant.
\end{proof}

\section{Hamiltonian actions}\label{sec:hamiltonian}

Let $\gS$ be a degree $1$ symplectic $Q$-manifold; that is, $\gS$ is a ($\Z$-)graded manifold equipped with a nondegenerate degree $-1$ Poisson structure and a degree $1$ Poisson vector field $Q_\gS$ such that $Q_\gS^2 = 0$. Suppose that a homotopy Poisson Lie group $(\gG, \phi)$ has a Hamiltonian right action on $\gS$ with action map $\sigma: \gS \times \gG \to \gS$ and equivariant moment map $\mu: \gS \to \LiegG^*[1]$.  

Let $L$ and $R$ be the maps from $T^*[1]\gG$ to $\LiegG^*[1]$ given by left- and right-translation, respectively. Specifically, if $\linv{v}$ and $\rinv{v}$ denote, respectively, the left- and right-invariant multivector fields associated to $v \in \bigS(\LiegG[-1]) \subseteq C^\infty(\LiegG^*[1])$, then $L$ and $R$ are given by $R^* v = \rinv{v}$ and $L^* v = \linv{v}$.

We identify $T^*[1]\gG$ with $\LiegG^*[1] \times \gG$ via right-translation, i.e.\ by identifying $R$ with the projection map $\LiegG^*[1] \times \gG \to \LiegG^*[1]$. Using this identification, we may lift the moment map to a map $\tilde{\mu} \defequal (\mu \times 1): \gS \times \gG \to T^*[1]\gG$. The equivariance property of $\mu$ is then equivalent to the commutativity of the diagram
\begin{equation}\label{diag:equivariance}
 \xymatrix{\gS \times \gG \ar^{\tilde{\mu}}[r] \ar_{\sigma}[d] & T^*[1]\gG \ar^{L}[d] \\ \gS \ar^{\mu}[r] & \LiegG^*[1]}.
\end{equation}

We construct a vector field $\Phi$ on $\gS \times \gG$, given by $\Phi(f) = Q_\gS (f)$ for $f \in C^\infty(\gS)$ and $\Phi(\alpha) = \tilde{\mu}^* d_\phi \alpha$ for $\alpha \in C^\infty(\gG)$.  On the right side of the latter equation, $\alpha$ is viewed as a fiberwise-constant function on $T^*[1]\gG$ (or, equivalently, as a $0$-vector field on $\gG$). 

\begin{lemma}
 If the moment map $\mu$ is a $Q$-manifold morphism (that is, if $Q_\gS$ is $\mu$-related to $\hat{d}_\phi$), then $\Phi^2 = 0$.
\end{lemma}
\begin{proof}
Let $\tau$ denote the projection map from $\gS \times \gG$ to $\gS$.  Observe that the diagram 
\begin{equation}
 \xymatrix{\gS \times \gG \ar^{\tilde{\mu}}[r] \ar_{\tau}[d] & T^*[1]\gG \ar^{R}[d] \\ \gS \ar^{\mu}[r] & \LiegG^*[1]}
\end{equation}
presents $\gS \times \gG$ as the fiber product $\gS \bitimes{\mu}{R} T^*[1]\gG$. The multiplicativity of $\phi$ implies that $R$ is a $Q$-manifold morphism \cite{me:qgpd}. If $\mu$ is also a $Q$-manifold morphism, then it follows that the product vector field $(Q_\gS, d_\phi) \in \vect(\gS \times T^*[1]\gG)$ is tangent to the fiber product. It is immediate from the definition of $\Phi$ that, in this case, the restriction of $(Q_\gS, d_\phi)$ to the fiber product is equal to $\Phi$. The equation $\Phi^2=0$ is then a consequence of the fact that $(Q_\gS, d_\phi)^2 = (Q_\gS^2, d_\phi^2) = 0$.
\end{proof}

Thus, if $\mu$ is a $Q$-manifold morphism, then $\Phi$ gives $\gS \times \gG$ the structure of a $Q$-manifold.

\begin{dfn}\label{dfn:qham}
The action of $\gG$ on $\gS$ is called \emph{$Q$-Hamiltonian} if the moment map $\mu$ and the action map $\sigma$ are both $Q$-manifold morphisms.
\end{dfn}

\begin{thm}\label{thm:reduction}
Let $\gG$ be a flat homotopy Poisson Lie group with a $Q$-Hamiltonian action on a degree $1$ symplectic $Q$-manifold $\gS$. If $0$ is a regular value of $\mu$ and the action of $\gG$ on  $\mu^{-1}(0)$ is free, then the homological vector field on $\gS$ descends to a symplectic homological vector field on $\mu^{-1}(0)/\gG$.
\end{thm}
\begin{proof}
Since $0$ is a regular value, the ideal $\cI$ of functions on $\gS$ that vanish on $\mu^{-1}(0)$ is generated by $\{ \mu^*v \suchthat v \in \LiegG[-1] \}$. The equivariance property of $\mu$ implies that $\{\mu^*v, \mu^*w\} = \mu^*[v,w]$ for all $v,w \in \LiegG[-1]$. Therefore $\cI$ is closed under the Poisson bracket, which means that $\mu^{-1}(0)$ is a coisotropic submanifold of $\gS$.

Let $v \in \LiegG[-1]$. Then $Q_\gS (\mu^* v) = \mu^* (\hat{d}_\phi v)$, since $\mu$ is assumed to be a $Q$-manifold morphism. The flatness of $\phi$ ensures that $\hat{d}_\phi v$ does not have a constant component, so $\hat{d}_\phi v$ vanishes at $0$, and therefore $Q_\gS (\mu^* v)$ vanishes on $\mu^{-1}(0)$. This proves that $\cI$ is $Q_\gS$-invariant.

If $\gG$ is connected, then the algebra of functions on the symplectic quotient is $N(\cI)/\cI$, where $N(\cI)$ is the Poisson normalizer of $\cI$. The fact that $\cI$ is $Q_\gS$-invariant implies that $N(\cI)$ is $Q_\gS$-invariant, and it follows that $Q_\gS$ descends to $N(\cI)/\cI$. 
The descended vector field inherits the properties of being symplectic and homological from $Q_\gS$.

In the case where $\gG$ is not connected, we need to furthermore show that the restriction of $Q_\gS$ to $\mu^{-1}(0)$ descends to the global quotient. For this, we consider the restriction of the action map $\sigma$ to $\sigma_0 : \mu^{-1}(0) \times \gG \to \mu^{-1}(0)$. The restriction of the vector field $\Phi$ to $\mu^{-1}(0) \times \gG$ is simply the product vector field $(Q_\gS, \phi_1)$, where $\phi_1 \in \vect(\gG)$ is the vector field component of $\phi$. The $Q$-Hamiltonian property of the action implies that $\sigma_0$ is a morphism of $Q$-manifolds. Thus, we may apply Lemma \ref{lemma:quotient} to the action of $(\gG, \phi_1)$ on $(\mu^{-1}(0), Q_\gS)$ to conclude that $Q_\gS$ descends to the quotient.
\end{proof}

\begin{remark}
	In the proof of Theorem \ref{thm:reduction}, the condition that $\sigma$ be a $Q$-manifold morphism was not needed in the case where $\gG$ is connected. The reason for this is that, when $\gG$ is connected, it can be shown that the condition that $\mu$ be a $Q$-manifold morphism implies the same for $\sigma$.
\end{remark}

\section{Homotopy Poisson actions revisited}\label{sec:pinftyact2}

In ordinary symplectic geometry, a standard example of Hamiltonian action is that of cotangent lift. Specifically, if $G$ is a Lie group with a free and proper action on a manifold $M$, then there is a natural way to lift the action of $G$ on $M$ to a Hamiltonian action of $G$ on $T^*M$, where the symplectic quotient may be identified with $T^*(M/G)$. In a sense, this allows us to view ordinary quotients as examples of symplectic quotients. 

In this section, we interpret quotients in the homotopy Poisson category in terms of symplectic quotients, via a supergeometric version of the cotangent lift construction. Throughout this section, let $(\gM, \pi)$ be a homotopy Poisson manifold, and let $(\gG, \phi)$ be a flat homotopy Poisson Lie group with a free and proper right homotopy Poisson action $\sigma: \gM \times \gG \to \gM$. 

First, we will describe the shifted cotangent lift action of $\gG$ on $T^*[1]\gM$.

Consider the pullback bundle $\sigma^*(T^*[1]\gM)$. We may make the identification $C^\infty(\sigma^*(T^*[1]\gM)) = C^\infty(\gM \times \gG) \otimes_\sigma \multivect(\gM)$. There is a natural push-forward map $\sigma_*: \multivect(\gM \times \gG) \to C^\infty(\sigma^*(T^*[1]\gM))$, satisfying the property that $\sigma_* Z = 1 \otimes X$ whenever $Z \in \multivect(\gM \times \gG)$ is $\sigma$-related to $X \in \multivect(\gM)$. The map $\sigma_*$ has a right inverse $\iota: C^\infty(\sigma^*(T^*[1]\gM)) \to \multivect(\gM \times \gG)$, whose image is the space of horizontal multivector fields. Thus, if $Z$ is a horizontal multivector field on $\gM \times \gG$, then $\iota \circ \sigma_* Z = Z$.

Let $\tilde{\sigma}$ and $\tilde{\tau}$ be maps from $\sigma^*(T^*[1]\gM)$ to $T^*[1]\gM$, given by $\tilde{\sigma}^*X = 1 \otimes X$ and $\tilde{\tau}^*X = \sigma_*(X,0)$ for $X \in \multivect(M)$. Letting $\tau$ be the projection map from $\gM \times \gG$ to $\gG$, we have that the following diagrams commute:
\begin{align*}
 \xymatrix{\sigma^*(T^*[1]\gM) \ar^{\tilde{\tau}}[r] \ar[d] & T^*[1]\gM \ar[d] \\ \gM \times \gG \ar^{\tau}[r] & \gM} && \xymatrix{\sigma^*(T^*[1]\gM) \ar^{\tilde{\sigma}}[r] \ar[d] & T^*[1]\gM \ar[d] \\ \gM \times \gG \ar^{\sigma}[r] & \gM}
\end{align*}
We may identify $\sigma^*(T^*[1]\gM)$ with $T^*[1]\gM \times \gG$ by taking $\tilde{\tau}$ to be the projection map, and we may then interpret $\tilde{\sigma}$ as the action map for a right action of $\gG$ on $T^*[1]\gM$.

\begin{dfn}
 The action of $\gG$ on $T^*[1]\gM$ associated to the map $\tilde{\sigma}:\sigma^*(T^*[1]\gM) = T^*[1]\gM \times \gG \to T^*[1]\gM$ is the \emph{(shifted) cotangent lift} of $\sigma$.
\end{dfn}

We emphasize that this cotangent lift construction is simply a supergeometric analogue of the ordinary cotangent lift construction. In particular, the homotopy Poisson structures on $\gG$ and $\gM$ do not play any role in the construction.

Next, we will show that $\tilde{\sigma}$ is a Hamiltonian action, where the moment map arises from the infinitesimal action $\LiegG \to \vect(\gM)$. 

Let $\rho: \LiegG \to \vect(\gM)$ be the infinitesimal action associated to $\sigma$, characterized by the property that $(0,\linv{v}) \in \vect(\gM \times \gG)$ is $\sigma$-related to $\rho(v)$ for $v \in \LiegG$. Similarly, the infinitesimal action associated to $\tilde{\sigma}$ is a map $\tilde{\rho} : \LiegG \to \vect(T^*[1]\gM)$, characterized by the property that $\tilde{v}$ is $\tilde{\sigma}$-related to $\tilde{\rho}(v)$ for $v \in \LiegG$, where $\tilde{v} \in \vect(\sigma^*(T^*[1]\gM))$ is determined by the equations
\begin{align}
 \tilde{v}(\tilde{\tau}^*X) &= 0, \label{eqn:tildev1} \\
\tilde{v}(\alpha \otimes 1) &= \linv{v}(\alpha) \otimes 1 \label{eqn:tildev2}
\end{align}
for $X \in \multivect(\gM)$ and $\alpha \in C^\infty(\gG)$.
\begin{lemma}\label{lemma:momentlift}
 $\tilde{\rho}(v) = [\rho(v), \cdot]$ for all $v \in \LiegG$.
\end{lemma}
\begin{proof}
 First, we claim that 
\begin{equation}\label{eqn:tildev3}
 \tilde{v}(\theta) = \sigma_* [(0,\linv{v}), \iota(\theta)]
\end{equation}
 for all $\theta \in C^\infty(\sigma^*(T^*[1]\gM))$. Indeed, for $X \in \multivect(\gM)$, we have
\begin{equation*}
 \begin{split}
  \sigma_* [(0,\linv{v}), \iota(\tilde{\tau}^*X)] &=  \sigma_* [(0,\linv{v}), \iota \circ \sigma_*(X,0)] \\
&=  \sigma_* [(0,\linv{v}), (X,0)] \\
&= 0,
 \end{split}
\end{equation*}
and for $\alpha \in C^\infty(\gG)$, we have
\begin{equation*}
 \begin{split}
  \sigma_* [(0,\linv{v}), \iota(\alpha \otimes 1)] &=  \sigma_* [(0,\linv{v}), (0,\alpha)] \\
&=  \sigma_* (0,\linv{v}(\alpha)) \\
&= \linv{v}(\alpha) \otimes 1.
 \end{split}
\end{equation*}
Thus we see that \eqref{eqn:tildev3} is valid, since it agrees with \eqref{eqn:tildev1} and \eqref{eqn:tildev2}. 

To prove the lemma, we need to show that $\tilde{v}$ is $\tilde{\sigma}$-related to $[\rho(v), \cdot]$ for $v \in \LiegG$. Using \eqref{eqn:tildev3} and the fact that $\iota(1 \otimes X)$ is $\sigma$-related to $X$, we have
\begin{equation*}
 \begin{split}
  \tilde{v}(\tilde{\sigma}^*X) &= \tilde{v}(1 \otimes X) \\
&= \sigma_*[(0, \linv{v}), \iota(1 \otimes X)] \\
&= 1 \otimes [\rho(v),X] \\
&= \tilde{\sigma}^*([\rho(v),X])
 \end{split}
\end{equation*}
for $X \in \multivect(\gM)$, as desired.
\end{proof}

An immediate consequence of Lemma \ref{lemma:momentlift} is the following:
\begin{cor}
 The cotangent lift action $\tilde{\sigma}$ is Hamiltonian, with moment map $\mu: T^*[1]\gM \to \LiegG^*[1]$ given by $\mu^*(v) = \rho(v)$ for $v \in \LiegG[-1]$.
\end{cor}

Finally, we will show, using the fact that the action of $\gG$ on $\gM$ is homotopy Poisson, that the cotangent lift action is $Q$-Hamiltonian, and therefore the hypotheses of Theorem \ref{thm:reduction} are satisfied.

Following the construction of \S\ref{sec:hamiltonian}, we lift the moment map to a map $\tilde{\mu}: \sigma^*(T^*[1]\gM) \to T^*[1]\gG$, which in this case is given by $\tilde{\mu}^*\theta = \sigma_*(0,\theta)$ for $\theta \in \multivect(\gG)$. Then we define the vector field $\Phi \in \vect(\sigma^*(T^*[1]\gM))$ by
\begin{align*}
 \Phi(\tilde{\tau}^*X) &= \tilde{\tau}^*(d_\pi X),\\
\Phi(\alpha \otimes 1) &= \tilde{\mu}^* d_\phi \alpha,
\end{align*}
for $X \in \multivect(\gM)$ and $\alpha \in C^\infty(\gG)$.

\begin{lemma}\label{lemma:philift}
$\Phi = \sigma_* \circ d_{(\pi,\phi)} \circ \iota$.
\end{lemma}
\begin{proof}
It suffices to check that $\Phi$ and $\sigma_* \circ d_{(\pi,\phi)} \circ \iota$ agree on functions of the form $\tilde{\tau}^*X$ and $\alpha \otimes 1$ for $X \in \multivect(\gM)$ and $\alpha \in C^\infty(\gG)$. Thus we compute
\begin{equation*}
 \begin{split}
   \Phi(\tilde{\tau}^*X) &= \sigma_*(d_\pi X,0)\\
&= \sigma_* \circ d_{(\pi,\phi)}(X,0) \\
&= \sigma_* \circ d_{(\pi,\phi)} \circ \iota \circ \sigma_*(X,0) \\
&= \sigma_* \circ d_{(\pi,\phi)} \circ \iota \circ \tilde{\tau}^*X
 \end{split}
\end{equation*}
and
\begin{equation*}
 \begin{split}
\Phi(\alpha \otimes 1) &= \sigma_* (0, d_\phi \alpha)\\
&= \sigma_* \circ d_{(\pi,\phi)} \alpha \\
&= \sigma_* \circ d_{(\pi,\phi)} \circ \iota(\alpha \otimes 1).\qedhere
 \end{split}
\end{equation*}
\end{proof}

\begin{prop}\label{prop:quotientham}
The cotangent lift $\tilde{\sigma} : T^*[1]\gM \times \gG \to T^*[1]\gM$ is a $Q$-Hamiltonian action.
\end{prop}
\begin{proof}
Proposition \ref{prop:infaction} implies that the moment map $\mu$ is a $Q$-manifold morphism. It remains to show that $\tilde{\sigma}$ is a $Q$-manifold morphism.

 The fact that $\sigma$ is a homotopy Poisson map means that, if $Z \in \multivect(\gM \times \gG)$ is $\sigma$-related to $X \in \multivect(\gM)$, then $d_{(\pi,\phi)} Z$ is $\sigma$-related to $d_\pi X$. For all $X \in \multivect(\gM)$, we have that $\iota(1 \otimes X)$ is $\sigma$-related to $X$. Therefore $d_{(\pi,\phi)} \circ \iota(1 \otimes X)$ is $\sigma$-related to $d_\pi X$, or in other words,
\begin{equation*}
 \sigma_* \circ d_{(\pi,\phi)} \circ \iota(1 \otimes X) = 1 \otimes d_\pi X.
\end{equation*}
Using Lemma \ref{lemma:philift}, we then have $\Phi(1 \otimes X) = 1 \otimes d_\pi X$, which may be rewritten as $\Phi(\tilde{\sigma}^*X) = \tilde{\sigma}^*(d_\pi X)$. Thus we conclude that $\Phi$ is $\tilde{\sigma}$-related to $d_\pi$.
\end{proof}

Proposition \ref{prop:quotientham} allows us to apply Theorem \ref{thm:reduction}, and the following result is a straightforward consequence.

\begin{thm}\label{thm:reductionaction}
 Let $\gM$ be a homotopy Poisson manifold, and let $\gG$ be a flat homotopy Poisson Lie group with a free and proper homotopy Poisson action on $\gM$. Then the homotopy Poisson structure on the quotient $\gM/\gG$ is given by the induced homological vector field on the symplectic quotient $T^*[1]\gM // \gG = T^*[1](\gM/\gG)$.
\end{thm}

\begin{example}
 Let $M$ be a Poisson manifold, and let $G$ be a Poisson Lie group with a free and proper Poisson action on $M$. Then Theorem \ref{thm:reductionaction} says that the reduced symplectic $NQ$-manifold $T^*[1]M//G$ corresponds to the quotient Poisson manifold $M/G$.
Thus we have an interpretation of Poisson quotients in terms of reduction of degree $1$ symplectic $NQ$-manifolds.
\end{example}

\section{Homotopy Poisson structures of degree $n$}\label{sec:higher}

The framework of ``reduction by homotopy Poisson Lie group action'' described in \S\ref{sec:hamiltonian} is specific to degree $1$ symplectic $Q$-manifolds. In order to deal with symplectic $Q$-manifolds of higher degree, we need to consider analogues of homotopy Poisson Lie groups in higher degrees. In this section, we briefly give definitions and state results for homotopy Poisson structures of degree $n$. We omit proofs, since those for the case $n=1$ generalize in a straightforward way.

A \emph{homotopy Poisson algebra structure of degree $n$} on $\gA$ is defined exactly as in Definition \ref{dfn:palg}
except that $\shiftgA$  is now defined to be the degree $n$ suspension $\gA[n]$, so that $\shiftgA_i = \gA_{i+n}$. If $\beta_\ell$ is thought of as a bracket on $\gA$, then $\beta_\ell$ is of degree $1+(1-\ell)n$.

A \emph{homotopy Poisson manifold of degree $n$} is a graded manifold $\gM$ whose algebra of functions $C^\infty(\gM)$ is equipped with a degree $n$ homotopy Poisson algebra structure of finite type. A degree $n$ homotopy Poisson structure on a graded manifold $\gM$ may be described by a degree $n+1$ element $\pi \in C^\infty(T^*[n]\gM)$ such that $\{\pi,\pi\}=0$. Here, the bracket is the canonical degree $-n$ Poisson bracket on $T^*[n]\gM$. The map $\pi \mapsto d_\pi$, where $d_\pi = \{\pi,\cdot\}$ gives a bijection between degree $n$ homotopy Poisson structures and polynomial symplectic homological vector fields on $T^*[n]\gM$.

A \emph{homotopy Poisson Lie group of degree $n$} is a graded Lie group equipped with a homotopy Poisson manifold structure such that the multiplication map is a homotopy Poisson morphism. A \emph{homotopy Lie bialgebra of degree $n$} is a graded Lie algebra $\LiegG$ equipped with a differential on $\bigS(\LiegG[-n])$ that is a derivation of the degree $-n$ Schouten bracket. Such a differential on $\bigS(\LiegG[-n])$ may equivalently be viewed as a polynomial homological Poisson vector field on $\LiegG^*[n]$, or an $\linfty$-algebra structure on $\LiegG^*[n-1]$ that is compatible with the Lie algebra structure on $\LiegG$.

We remark that a $Q$-group (resp.\ DGLA) may be viewed as a homotopy Poisson Lie group (resp.\ homotopy Lie bialgebra) of any degree.

Let $\gS$ be a degree $n$ symplectic $Q$-manifold, and let $\gG$ be a flat degree $n$ homotopy Poisson Lie group with a Hamiltonian action on $\gS$. The moment map $\mu: \gS \to \LiegG^*[n]$ may be lifted to a map $\tilde{\mu}: \gS \times \gG \to T^*[n]\gG$, which may then be used to define a vector field $\Phi$ on $\gS \times \gG$, as in \S\ref{sec:hamiltonian}. As in Definition \ref{dfn:qham}, the action is called $Q$-Hamiltonian if the moment map and action map are both $Q$-manifold morphisms. Then Theorem \ref{thm:reduction} generalizes as follows.
\begin{thm}\label{thm:reductionhigher}
Let $\gG$ be a flat degree $n$ homotopy Poisson Lie group with a $Q$-Hamiltonian action on a degree $n$ symplectic $Q$-manifold $\gS$. If $0$ is a regular value of $\mu$ and the action of $\gG$ on $\mu^{-1}(0)$ is free, then the homological vector field on $\gS$ descends to a symplectic homological vector field on $\mu^{-1}(0)/\gG$.
\end{thm}

We conclude by describing two examples of degree $2$ homotopy Lie bialgebras.

\begin{example}[Left-central Courant algebras]
A \emph{Courant algebra} \cite{bcg} over a Lie algebra $\mathfrak{g}$ is a vector space $\mathfrak{a}$ equipped with a bilinear bracket $\llbracket \cdot, \cdot \rrbracket$ and a map $p: \mathfrak{a} \to \mathfrak{g}$ such that, for all $a_1, a_2, a_3 \in \mathfrak{a}$,
\begin{enumerate}
\item $\llbracket a_1, \llbracket a_2, a_3 \rrbracket \rrbracket = \llbracket \llbracket a_1,  a_2\rrbracket, a_3  \rrbracket + \llbracket a_2, \llbracket a_1, a_3 \rrbracket \rrbracket$,
\item $p( \llbracket a_1, a_2 \rrbracket) = [p(a_1), p(a_2)]$.
\end{enumerate}
A Courant algebra $\mathfrak{a}$ is \emph{left-central} if \ $\llbracket h, a \rrbracket = 0$ for all $h \in \ker p$, $a \in \mathfrak{a}$.

Let $\mathfrak{a}$ be a left-central Courant algebra, and let $\mathfrak{h} \defequal \ker p$. We may form a graded vector space $\tilde{\mathfrak{g}} \defequal \mathfrak{h}[2] \oplus \mathfrak{a}[1] \oplus \mathfrak{g}$, where elements of $\mathfrak{h}$ and $\mathfrak{a}$ are viewed as being of degree $-2$ and $-1$, respectively. The bracket on $\mathfrak{a}$ induces a graded Lie bracket $\{\cdot, \cdot\}$ on $\tilde{\mathfrak{g}}$, defined as
\begin{equation}\label{eqn:ecabracket}\begin{split}
&\{(h_1,a_1,g_1, h_2,a_2,g_2)\} = \\
&= \left( \llbracket b_1, h_2 \rrbracket + \llbracket a_1, a_2 \rrbracket + \llbracket a_2, a_1 \rrbracket - \llbracket b_2, h_1 \rrbracket, \llbracket b_1, a_2 \rrbracket - \llbracket b_2, a_1 \rrbracket, [g_1,g_2] \right),
\end{split}
\end{equation}
where $b_1, b_2 \in \mathfrak{a}$ are such that $p(b_i) = g_i$.  The left-central property of $\mathfrak{a}$ implies that (\ref{eqn:ecabracket}) is well-defined.

The exact sequence
\begin{equation*}
\xymatrix{0 \ar[r] & \mathfrak{h} \ar[r] & \mathfrak{a} \ar[r] & \mathfrak{g} \ar[r] & 0}
\end{equation*}
determines a differential on $\tilde{\mathfrak{g}}$. This differential is a derivation of the Lie bracket, so $\tilde{\mathfrak{g}}$ is a DGLA.

Let $\gS$ be the degree $2$ symplectic $NQ$-manifold corresponding to a Courant algebroid $E$. Then an infinitesimal $Q$-Hamiltonian action of $\tilde{\mathfrak{g}}$ on $\gS$ may be shown to be equivalent to what Bursztyn, Cavalcanti, and Gualtieri \cite{bcg} called an ``extended action with moment map'' of $\mathfrak{a}$ on $E$. In this case, their notion of moment map reduction coincides with $Q$-Hamiltonian reduction \cite{bcmz}.
\end{example}

\begin{example}[Matched pairs]
 Let $\tilde{\mathfrak{g}} = \mathfrak{h} [1] \oplus \mathfrak{g}$ be a graded Lie algebra; in other words, $\mathfrak{g}$ is a Lie algebra equipped with an action on the vector space $\mathfrak{h}$. The action can be dualized to an action of $\mathfrak{g}$ on $\mathfrak{h}^*$.

A quadratic degree $2$ homotopy Lie bialgebra structure on $\tilde{\mathfrak{g}}$ is given by a graded Lie algebra structure on $\tilde{\mathfrak{g}}^*[1] = \mathfrak{h}^* \oplus \mathfrak{g}^*[1]$, which is equivalent to a Lie algebra structure on $\mathfrak{h}^*$ and an action of $\mathfrak{h}^*$ on $\mathfrak{g}^*$. This action may be dualized to an action of $\mathfrak{h}^*$ on $\mathfrak{g}$.

There is a compatibility condition between the graded Lie algebra structures on $\tilde{\mathfrak{g}}$ and $\tilde{\mathfrak{g}}^*[1]$. Expressed in terms of the actions of $\mathfrak{g}$ and $\mathfrak{h}^*$ on each other, the compatibility condition is exactly that of a \emph{matched pair} \cite{majid:matched1,majid:matched2} of Lie algebras. Thus, Lie algebra matched pairs form a special case of degree $2$ homotopy Lie bialgebras.
\end{example}

\bibliographystyle{abbrv}
\bibliography{bibliography}

\end{document}